\documentclass[11pt]{article}

\usepackage{morefloats}
\usepackage{amsfonts}
\usepackage{color}
\usepackage{multirow}
\usepackage{latexsym}
\usepackage{amsmath,amssymb,amsthm,graphicx}
\usepackage{dsfont}
\usepackage{extarrows}
\usepackage{multirow}
\usepackage{hyperref}

\numberwithin{equation}{section}

\newtheoremstyle{thm}
{9pt}
{9pt}
{\itshape}
{}
{\bfseries}
{.}
{ }
{}
\theoremstyle{thm}

\newtheorem{theorem}{Theorem}[section]

\newtheoremstyle{def}
{9pt}
{9pt}
{}
{}
{\bfseries}
{.}
{ }
{}
\theoremstyle{def}

\newcommand{\R}{\mathbb{R}} 

\renewcommand{\footnoterule}{%
	\kern -3.5pt
	\hrule width \textwidth height 1pt
	\kern 3.5pt
}

\makeatletter
\def\blfootnote{\xdef\@thefnmark{}\@footnotetext}
\makeatother

\begin{document}

\title{\bf  On the eigenvalues associated with the limit null distribution of the Epps-Pulley test of normality}


\author{Bruno Ebner and Norbert Henze}

\date{\today}
\maketitle

\blfootnote{ {\em MSC 2010 subject
classifications.} Primary 62F03 Secondary 65C60; 65R20}
\blfootnote{
{\em Key words and phrases}: Test for normality; integral operator; Fredholm determinant }

\begin{abstract}
The Shapiro--Wilk test (SW) and the Anderson--Darling test (AD) turned out to be strong procedures for testing for normality. They are joined by a class
of tests for normality proposed by Epps and Pulley that, in contrary to SW and AD, have been extended by Baringhaus and Henze to yield easy-to-use
affine invariant and universally consistent tests for normality in any dimension. The limit null distribution of the Epps--Pulley test involves
a sequences of eigenvalues of a certain integral operator induced by the covariance kernel of the limiting Gaussian process. We solve the associated integral equation and present the corresponding eigenvalues.
\end{abstract}
\noindent
\section{Introduction}\label{sec:Intro}
Let $X,X_1, X_2 \ldots$ be a sequence of independent and identically distributed (i.i.d)  random variables with unkown distribution. To test the hypothesis
$H_0$ that the distribution of $X$ is some unspecified normal distribution, there is a myriad of testing procedures, among which the tests of Shapiro--Wilk (SW) and
Anderson--Darling (AD) deserve special mention, see, e.g., the monographs \cite{DAS:1986,T:2002}. There is, however,
a further test which was proposed by Epps and Pulley (\cite{EP:1983}). This test, which is based on the empirical characteristic function, comes as a serious competitor to SW and AD, as shown
in simulation studies (see, e.g., \cite{BDH:1989,BE:2020a}). Baringhaus and Henze (\cite{BH:1988}) extended the approach of Epps and Pulley to test for normality in any dimension.
By now, the BHEP-test (an acronym coined by S. Cs{\"o}rg\H{o} \cite{CS:1989}) after earlier developers of the idea) is known to be an affine-invariant and universally consistent
test of normality in any dimension, and limit distributions of the test statistic have been obtained under $H_0$ as well as under fixed and contiguous alternatives to normality
(see the review article \cite{EH:2020}).  In this paper, we revisit the limit null distribution of the Epps-Pulley-test statistic in the univariate case.
The test statistic involves a positive tuning parameter $\beta$, and, based on $X_1,\ldots,X_n$, is denoted by $T_{n,\beta}$. It is given by
\[
T_{n,\beta} = n \int_{-\infty}^\infty \Big{|} \psi_n(t) - {\rm e}^{-t^2/2}\Big{|}^2 \, \varphi_\beta(t)\, {\rm d}t,
\]
where $\psi_n(t) = n^{-1}\sum_{j=1}^n \exp\big({\rm i}tY_{n,j}\big)$ is the empirical characteristic function of
the {\em scaled residuals} $Y_{n,1}, \ldots, Y_{n,n}$. Here, $Y_{n,j} = S_n^{-1} (X_j - \overline{X}_n)$, $j=1,\ldots,n$,
and $\overline{X}_n = n^{-1} \sum_{j=1}^n X_j$, $S_n^2 = n^{-1}\sum_{j=1}^n (X_j- \overline{X}_n)^2$ denote the sample mean and
the sample variance of $X_1,\ldots,X_n$, respectively. Moreover,
\[
\varphi_\beta(t) = \frac{1}{\beta \sqrt{2\pi}} \exp \Big( - \frac{t^2}{2\beta^2 } \Big), \quad t \in \mathbb{R},
\]
is the density of the centred normal distribution with variance $\beta^2$. A closed-form expression of $T_{n,\beta}$ that is amenable to
computational purposes is
\begin{equation*}\label{deftnb}
T_{n,\beta} = \frac{1}{n} \sum_{j,k=1}^n \exp \! \bigg( \! \! - \frac{\beta^2}{2}\big(Y_{n,j}\! - \! Y_{n,k}\big)^2 \! \bigg) \!
 - \frac{2}{\sqrt{1 \! + \! \beta^2}} \sum_{j=1}^n \exp \! \bigg( \! \! - \frac{\beta^2 Y_{n,j}^2 }{2(1\! + \! \beta^2)} \! \bigg)
 + \frac{n}{\sqrt{1\! + \! 2\beta^2}}.
\end{equation*}
The limit null distribution of $T_{n,\beta}$, as $n \to \infty$, is that of
\[
T_\infty :=  \int_{-\infty}^\infty Z^2(t) \, \varphi_\beta(t) \, \text{d}t.
\]
Here, $Z(\cdot)$ is a centred Gaussian element of the Hilbert space $\text{L}^2 = \text{L}^2(\mathbb{R},{\cal B},\varphi_\beta(t)\text{d}t)$ of Borel-measurable
real-valued functions that are square-integrable with respect to $\varphi_\beta(t)\text{d}t$, and the covariance function of $Z(\cdot)$ is given by
\begin{equation}\label{eq:KHW}
K(s,t)=\exp\left(-\frac{(s-t)^2}2\right)-\left(1+st+\frac{(st)^2}2\right)\exp\left(-\frac{s^2}2-\frac{t^2}2\right),\quad s,t\in\R
\end{equation}
(see \cite{HW:1997}). The kernel $K$ is the starting point of this paper.
Writing $\sim$ for equality in distribution,  it is well-known that
\[
T_\infty \sim \sum_{j=1}^\infty \lambda_j N_j^2,
\]
where $\lambda_1, \lambda_2, \ldots $ is the sequence on nonzero eigenvalues associated with the integral operator $\mathbb{A}: \text{L}^2 \rightarrow \text{L}^2$
defined by
\[
(\mathbb{A}f)(s) := \int_{-\infty}^\infty K(s,t) f(t) \varphi_\beta(t) \, \text{d} t,
\]
and $N_1, N_2, \ldots $ is a sequence of i.i.d. standard normal random variables.  In the next section, we obtain the eigenvalues of $\mathbb{A}$ by numerical methods. In Section \ref{sec:Acc} the sum of powers of the largest eigenvalues is compared to normalized cumulants. The difference should be close to 0 if the eigenvalues have been computed correctly. Section \ref{sec:Pearson} demonstrates that the results can be applied to fit a Pearson system of distributions, and that the fit is reasonable to approximate critical values of the Epps-Pulley test. The article ends by some concluding remarks.

\noindent
\section{Solution of a Fredholm integral equation}\label{sec:fredholm}
To obtain the values $\lambda_1, \lambda_2, \ldots$,  that determine the distribution of $T_\infty$, one has to solve the integral equation
\begin{equation*}
\int_{-\infty}^\infty K(s,t)f(t) \varphi_\beta (t)\,\text{d}t=\lambda f(s),\quad s\in\R.
\end{equation*}
In general, this is considered a hard problem,  and solutions for kernels associated with testing problems involving
composite hypotheses are very sparse, see \cite{S:1976,S:1977} for the classical tests of normality and exponentiality that are based on the empirical distribution function.
In what follows, we use a result of \cite{ZWRM:1997} to obtain the eigenvalues of $\mathbb{A}$ by a stable numerical method.
To this end, let
\begin{equation*}
K_0(s,t)=\exp\left(-\frac{(s-t)^2}2\right),\quad s,t\in\R,
\end{equation*}
and
\begin{eqnarray*}
\phi_1(s)&=&\frac{s^2}{\sqrt{2}}\exp\left(-\frac{s^2}2\right),\\
\phi_2(s)&=&s\exp\left(-\frac{s^2}2\right),\\
\phi_3(s)&=&\exp\left(-\frac{s^2}2\right),\quad s\in\R.
\end{eqnarray*}
Notice that
\begin{equation}\label{eq:kernelrep}
K(s,t)=K_0(s,t)-\sum_{j=1}^3\phi_j(s)\phi_j(t),\quad s,t\in\R.
\end{equation}
The first step is to solve the eigenvalue problem for the covariance kernel $K_0$, which corresponds to the limit distribution of the
a modified statistic $T_{n,\beta}^{(0)}$, which originates from $T_{n,\beta}$ by replacing $\psi_n(t)$ with $\psi_n^{(0)}(t) =  n^{-1}\sum_{j=1}^n \exp\big({\rm i}tX_j\big)$,
i.e., the problem is to test for standard normality and thus no estimation of parameters is involved. The associated eigenvalue problem, which
leads to the kernel $K_0$, was
 solved in the context of machine learning in Chapter 4 of \cite{ZWRM:1997}. Here, we use the formulation given in \cite{RW:2008}, Subsection 4.3.1. In our case, we have
\begin{equation*}
\lambda_k^{(0)}=\sqrt{\frac2{\sqrt{1+4\beta^2}+2\beta^2+1}}\left(\frac{2\beta^2}{\sqrt{\sqrt{1+4\beta^2}+2\beta^2+1}}\right)^k,\quad k=0,1,2,\ldots,
\end{equation*}
with corresponding normalized eigenfunctions
\begin{equation*}
\psi_k(x)= h_k \exp\left(-\left(\frac{\sqrt{\beta^{-2}+4}}{4\beta}-(4\beta)^{-1}\right)x^2\right) H_k\left(\left(\beta^{-4}+4\beta^{-2}\right)^{1/4}x/\sqrt{2}\right),\quad k=0,1,2,\ldots
\end{equation*}
(see also the errata to \cite{RW:2008} on the books homepage).
Here, $h_k^{-2} = (4\beta^2+1)^{-1/4} 2^k k!$, and $H_k(x)=(-1)^k\exp(x^2)\frac{\mbox{d}^k}{\mbox{d}x^k}\exp(-x^2)$ is the $k$th order Hermite polynomial.
To solve the eigenvalue problem of $\mathbb{A}$ figuring in \eqref{eq:KHW}, we adapt the methodology in \cite{S:1976}. Define
\begin{eqnarray*}
a_{j,k}&=&\int_{-\infty}^\infty \psi_j(x)\phi_k(x)\varphi_\beta(x)\;\mbox{d}x,\\
S_{k}(\lambda)&=&1+\sum_{j=0}^\infty\frac{a_{j,k}^2}{1/\lambda- \lambda_k^{(0)}},\\
S_{k,\ell}(\lambda)&=&\sum_{j=0}^\infty\frac{a_{j,k}a_{j,\ell}}{1/\lambda-\lambda_k^{(0)}},\quad\lambda>0.
\end{eqnarray*}
With this notation, we can formulate our main result.
\begin{theorem}\label{thm:Eigprob}
The eigenvalues of $\mathbb{A}$ are the reciprocals of the solutions $\lambda>0$ of the equation
\begin{equation}\label{eq:FredDet}
D(\lambda)=d(\lambda)S_2(\lambda)\left(S_{1}(\lambda)S_3(\lambda)-S_{1,3}^2(\lambda)\right)=0,
\end{equation}
where $d(\lambda)=\prod_{k=0}^\infty(1/\lambda-\lambda_k)$ is the Fredholm determinant connected to the eigenvalue problem of $K_0$. Moreover, none of the reciprocals of the eigenvalues $\lambda_k^{(0)}$ of $K_0$ solve equation \eqref{eq:FredDet}.
\end{theorem}
\begin{proof}
Since $a_{j,1}a_{j,2}=a_{j,2}a_{j,3}=0$ holds for all $j=0,1,2,\ldots$, we use Theorem 2.2 of \cite{S:1972}
to see that the Fredholm determinant for the eigenvalue problem takes the form
\begin{equation*}
D(\lambda)=d(\lambda)\det\left(\begin{array}{ccc} S_{1}(\lambda) & 0 & S_{1,3}(\lambda) \\ 0 & S_2(\lambda) & 0 \\ S_{1,3}(\lambda) & 0 & S_3(\lambda)\end{array}\right)=d(\lambda)S_2(\lambda)\left(S_{1}(\lambda)S_3(\lambda)-S_{1,3}^2(\lambda)\right).
\end{equation*}
Hence, the reciprocals of the roots of $D(\lambda)$ are the eigenvalues of $\mathbb{A}$. By direct calculation, it follows that $a_{j,2}=0$ if $j$ is even,
and we have $a_{j,1}=a_{j,3}=0$ if $j$ is odd. Consequently, none of the reciprocals of the eigenvalues $\lambda_k^{(0)}$ is a root
of $D(\lambda)$ and thus a solution of the eigenvalue problem associated with the kernel $K$.
\end{proof}
\noindent
According to Theorem \ref{thm:Eigprob}, the eigenvalues of $\mathbb{A}$ are the roots of
$S_2(\lambda)$ and of $S_{1}(\lambda)S_3(\lambda)-S_{1,3}^2(\lambda)$. The reciprocals of the roots have been obtained numerically,
and the first twenty eigenvalues are displayed in Table \ref{tab:EV} for different values of $\beta$. Note that, since these values tend to be very small,
the reciprocal approach used here leads to numerically stable procedures to find the roots of the Fredholm determinant.
\begin{table}[t!]
\centering
\begin{tabular}{r|rrrrr}

$\lambda_j\backslash \beta$ & 0.25 & 0.5 & 1 & 2 & 3 \\
  \hline
$\lambda_0$ & 4.07235E-04 & 1.01443E-02 & 7.42748E-02 & 1.54164E-01 & 1.59960E-01 \\
  $\lambda_1$ & 3.96229E-05 & 2.98027E-03 & 4.48104E-02 & 1.29257E-01 & 1.45877E-01 \\
  $\lambda_2$ & 8.87536E-07 & 2.13968E-04 & 8.41907E-03 & 4.99665E-02 & 7.56703E-02 \\
  $\lambda_3$ & 7.41169E-08 & 5.45396E-05 & 4.58684E-03 & 3.98239E-02 & 6.69664E-02 \\
  $\lambda_4$ & 2.36367E-09 & 5.42325E-06 & 1.07998E-03 & 1.70946E-02 & 3.68745E-02 \\
  $\lambda_5$ & 1.81032E-10 & 1.27337E-06 & 5.51939E-04 & 1.31547E-02 & 3.19372E-02 \\
  $\lambda_6$ & 6.72990E-12 & 1.46554E-07 & 1.45739E-04 & 6.00412E-03 & 1.82413E-02 \\
  $\lambda_7$ & 4.87430E-13 & 3.26023E-08 & 7.12110E-05 & 4.49725E-03 & 1.55292E-02 \\
  $\lambda_8$ & 1.97617E-14 & 4.08130E-09 & 2.01821E-05 & 2.14175E-03 & 9.10854E-03 \\
  $\lambda_9$ & 1.37638E-15 & 8.73898E-10 & 9.53839E-06 & 1.56980E-03 & 7.64324E-03 \\
  $\lambda_{10}$ & 5.90134E-17 & 1.15555E-10 & 2.83684E-06 & 7.71666E-04 & 4.57714E-03 \\
  $\lambda_{11}$ & 3.99302E-18 & 2.40495E-11 & 1.30684E-06 & 5.55566E-04 & 3.79317E-03 \\
  $\lambda_{12}$ & 1.78040E-19 & 3.30498E-12 & 4.02503E-07 & 2.79917E-04 & 2.31011E-03 \\
  $\lambda_{13}$ & 1.17813E-20 & 6.72882E-13 & 1.81702E-07 & 1.98509E-04 & 1.89295E-03 \\
  $\lambda_{14}$& 5.40732E-22 & 9.51525E-14 & 5.74665E-08 & 1.01995E-04 & 1.16876E-03 \\
  $\lambda_{15}$ & 3.51545E-23 & 1.90367E-14 & 2.55205E-08 & 7.13642E-05 & 9.46891E-04 \\
  $\lambda_{16}$ & 1.64986E-24 & 2.75201E-15 & 8.24056E-09 & 3.71939E-05 & 5.90055E-04 \\
  $\lambda_{17}$ & 1.05731E-25 & 5.42770E-16 & 3.61045E-09 & 2.56128E-05 & 4.70590E-04 \\
  $\lambda_{18}$ & 6.60890E-28 & 1.07411E-17 & 1.18543E-09 & 2.86884E-06 & 8.19769E-05 \\
  $\lambda_{19}$ & 7.53533E-29 & 3.77859E-18 & 5.13526E-10 & 3.67008E-06 & 1.22699E-05
\end{tabular}
\caption{Eigenvalues of $\mathbb{A}$ for different tuning parameters $\beta$, here E-j stands as usually for $10^{-j}$.}\label{tab:EV}
\end{table}

\section{Accuracy of the numerical solutions}\label{sec:Acc}
The accuracy of the values presented in Table 1 may  be judged by a comparison  with results of \cite{H:1990}. That paper gives the first
four cumulants of  the distribution of $T_\infty$ in the special case $\beta=1$. The $m$-th cumulant of $T_\infty$ is
\[
\kappa_m(\beta) = 2^{m-1} (m-1)! \sum_{j=1}^\infty \lambda_j^m = 2^{m-1} (m-1)! \int_{-\infty}^\infty K_{m}(x,x) \varphi_\beta(x) \, \text{d} x, \quad m \ge 1,
\]
where $K_1(x,y) := K(x,y)$ and
\[
K_m(x,y) = \int_{-\infty}^\infty K_{m-1}(x,z)K(z,y) \varphi_\beta(z) \, \text{d} z
\]
for $m \ge 2$ (see e.g., Chapter 5 of \cite{SW:1986}). We have
\begin{eqnarray}\label{sumev}
\kappa_1(1)  & = & \mathbb{E}(T_\infty) =  \sum_{j=1}^\infty \lambda_j = 1 - \frac{\sqrt{3}}{2} = 0.133974596\ldots,\\ \nonumber
\kappa_2(1)  & = & \mathbb{V}(T_\infty) = 2 \sum_{j=1}^\infty \lambda_j^2 = \frac{2\sqrt{5}}{5} + \frac{5}{6} - \frac{155 \sqrt{2}}{128} = 0.015236301\ldots
\end{eqnarray}
and thus
\begin{equation*}
\sum_{j=1}^\infty \lambda_j^2 = \frac{1}{\sqrt{5}} + \frac{5}{12} - \frac{155}{128\sqrt{2}} = 0.0076181509 \ldots
\end{equation*}
Furthermore,
\begin{eqnarray}\label{sumev2}
\kappa_3(1)  & = & \mathbb{E}(T_\infty-\kappa_1(1))^3 = 8 \sum_{j=1}^\infty \lambda_j^3 = 0.00400343\ldots,\\ \nonumber
\kappa_4(1)  & = & \mathbb{E}(T_\infty-\kappa_1(1))^4 = 48 \sum_{j=1}^\infty \lambda_j^4 = 0.001654655\ldots
\end{eqnarray}
and thus
\begin{equation*}
\sum_{j=1}^\infty \lambda_j^3 = 0.0005004285291\ldots\quad\mbox{and}\quad\sum_{j=1}^\infty \lambda_j^4 = 0.00003447197917\ldots
\end{equation*}
From Table 2, we see that the corresponding sums of the first 20 numerical values resp. of their squares and cubes agree approximately with the values figuring in \eqref{sumev} and \eqref{sumev2}, respectively, up to five significant digits in most cases.

The results of \cite{H:1990} have been partially generalized in \cite{HW:1997}, Theorem 2.3, for the first three cumulants and a fixed tuning parameter $\beta$,
and they thus lead to general formulae in the univariate case. For the sake of completeness, we restate the formulae of the first two cumulants here. For the first cumulant, we have
\begin{equation*}
\kappa_1(\beta) = 1-(2\beta^2+1)^{-1/2}\left[1+\frac{\beta^2}{2\beta^2+1}+\frac{3\beta^4}{(2\beta^2+1)^2}\right],
\end{equation*}
and the second cumulant is
{\small
\begin{eqnarray*}
\kappa_2(\beta) & = & \frac{2}{\sqrt{1+4\beta^2}}+\frac2{1+2\beta^2}\left[1+\frac{2\beta^4}{(1+2\beta^2)^2}+\frac{9\beta^8}{4(1+2\beta^2)^4}\right]\\
&&-\frac{4}{\sqrt{1+4\beta^2+3\beta^4}}\left[1+\frac{3\beta^4}{2(1+4\beta^2+3\beta^4)}+\frac{3\beta^8}{2(1+4\beta^2+3\beta^4)^2}\right].
\end{eqnarray*}
The formula for the third cumulant is found in \cite{HW:1997}, Theorem 2.3, for the case $d=1$.
Table \ref{tab:comp} exhibits the normalized cumulants, together with the corresponding sums of the first 20 eigenvalues taken from Table \ref{tab:EV}.
We stress that by now no formula for the fourth cumulant is known in the literature for general tuning parameter $\beta$.

\begin{table}[t]
\centering
\begin{tabular}{r|rrrrr}
 $\beta$ & 0.25 & 0.5 & 1 & 2 & 3 \\
  \hline
  $\sum_{j=0}^{19}\lambda_j$ & 4.47822E-04 & 1.34000E-02 & 1.33975E-01 & 4.19722E-01 & 5.83761E-01 \\
  $\sum_{j=0}^{19}\lambda_j^2$ & 1.67411E-07 & 1.11838E-04 & 7.61814E-03 & 4.50863E-02 & 6.02196E-02 \\
  $\sum_{j=0}^{19}\lambda_j^3$ & 6.75980E-11 & 1.04392E-06 & 5.00428E-04 & 6.01902E-03 & 8.02464E-03 \\
  $\sum_{j=0}^{19}\lambda_j^4$ & 2.75053E-14 & 1.06687E-08 & 3.44719E-05 & 8.52858E-04 & 1.16350E-03 \\
  \hline
  $\kappa_1(\beta)$ & 4.47822E-04 & 1.34000E-02 & 1.33975E-01 & 4.19753E-01 & 5.84700E-01 \\
  $\kappa_2(\beta)/2$ & 1.66500E-07 & 1.11838E-04 & 7.61814E-03 & 4.50863E-02 & 6.02202E-02 \\
  $\kappa_3(\beta)/8$ & $\ast$ & 1.07115E-06 & 5.00429E-04 & 6.01903E-03 & 8.02468E-03 \\
\end{tabular}
\caption{Sums over different powers of the first 20 eigenvalues and corresponding theoretical cumulants for different values of $\beta$. The entry denoted by $\ast$ could not be computed due to numerical instabilities.}\label{tab:comp}
\end{table}


\section{Pearson system fit for approximation of critical values}\label{sec:Pearson}
The first four cumulants  can directly be used in packages that implement the Pearson system of distributions (see Section 4.1 of \cite{JKB:1994}).
 In the statistical computing language \texttt{R} (see \cite{R:2021}), we use the package \texttt{PearsonDS} (see \cite{BK:2017})
 to approximate critical values of the Epps-Pulley test statistic. The Epps-Pulley test is implemented in the \texttt{R}-Package \texttt{mnt} (see \cite{BE:2020})
 by using the function \texttt{BHEP}. Table \ref{tab:CV} shows simulated empirical critical values of the Epps-Pulley statistic
 for sample sizes $n\in\{10,25,50,100,200\}$ and levels of significance $\alpha\in\{0.1,0.05,0.01\}$.
 For each combination of $n$ and $\beta$, the entries corresponding to different values of $\alpha$  are based on $10^6$ replications under the null hypothesis.
 Each entry in a row named '$\infty$' is the calculated ($1-\alpha$)-quantile of the fitted Pearson system using the cumulants given in Table \ref{tab:comp}.
 We conclude that, for larger sample sizes,  the simulated critical values are close to the approximated counterparts of the Pearson system.
 Moreover, we have corroborated the results of \cite{H:1990} for the special case $\beta=1$, and we have extended these results for general $\beta>0$.

\begin{table}[t]
\centering
\begin{tabular}{c|r|rrrrr}
$\alpha$& $n\backslash \beta$ & 0.25 & 0.5 & 1 & 2 & 3 \\
  \hline
\multirow{6}{*}{0.1} & 10  & 7.28E-04  & 0.0245 & 0.277 & 0.817 & 1.03 \\
                     & 25  & 9.58E-04  & 0.0289 & 0.288 & 0.814 & 1.04 \\
                     & 50  & 1.05E-03  & 0.0304 & 0.289 & 0.811 & 1.04 \\
                     & 100 & 1.10E-03  & 0.0310 & 0.290 & 0.812 & 1.04 \\
                     & 200 & 1.13E-03  & 0.0314 & 0.291 & 0.811 & 1.04 \\
                     & $\infty$ & 1.14E-03 & 0.0319 & 0.292 & 0.812 & 1.04 \\\hline
\multirow{6}{*}{0.05} & 10  & 1.06E-03 & 0.0343 & 0.355 & 0.99 & 1.22 \\
                      & 25  & 1.39E-03 & 0.0403 & 0.371 & 1.00 & 1.25 \\
                      & 50  & 1.51E-03 & 0.0420 & 0.374 & 1.01 & 1.25 \\
                      & 100 & 1.57E-03 & 0.0427 & 0.376 & 1.01 & 1.25 \\
                      & 200 & 1.60E-03 & 0.0429 & 0.378 & 1.01 & 1.25 \\
                      & $\infty$ & 1.61E-03 & 0.0429 & 0.379 & 1.01 & 1.25 \\\hline  
\multirow{6}{*}{0.01} &  10     & 1.91E-03 & 0.0589 & 0.543 & 1.39 & 1.65 \\
                      &  25     & 2.58E-03 & 0.0696 & 0.570 & 1.44 & 1.72 \\
                      &  50     & 2.75E-03 & 0.0711 & 0.575 & 1.45 & 1.74 \\
                      &  100    & 2.78E-03 & 0.0720 & 0.581 & 1.46 & 1.75 \\
                      &  200    & 2.78E-03 & 0.0717 & 0.585 & 1.46 & 1.75 \\
                      & $\infty$ & 2.74E-03 & 0.0700 & 0.585 & 1.46 & 1.74
\end{tabular}
\caption{Empirical critical values (simulated with $10^6$ replications)  and approximated critical values by the Pearson system for different
 levels of significance $\alpha$}\label{tab:CV}
\end{table}

\section{Conclusions}
   We have solved the eigenvalue problem of the integral operator associated with the covariance kernel $K$ of the limiting Gaussian process
   that occurs in the limit null distribution of the Epps-Pulley test statistic. In view of a comparison with the first three known cumulants from the literature,
   Table \ref{tab:comp} shows that the eigenvalues obtained by  numerical methods are very close to the corresponding theoretical values.
   In Section 5 of \cite{EH:2021}, the authors present a Monte Carlo based approximation method to find stochastic approximations of the eigenvalues.
   A comparison of Table \ref{tab:EV} and Table 1 of \cite{EH:2021} reveals that there are some significant differences for some values of $\beta$.
   This observation is of particular interest, since the largest eigenvalue is used in the derivations of approximate Bahadur efficiencies.
   Recent results concerning this topic for the Epps-Pulley test are presented in \cite{EH:2021} and, for other normality tests based on the empirical distribution function, in \cite{MNO:2021}.

   We finally point out the difficulties encountered if one tries to generalize our findings to the multivariate case, i.e. to obtain the eigenvalues associated with the limit null distribution of the BHEP test of multivariate normality, see \cite{BH:1988,HW:1997,HZ:1990}. The $d$-variate analog to the covariance kernel $K$ in \eqref{eq:KHW} is given in Theorem 2.1 of \cite{HW:1997}, namely, writing $\|\cdot\|$ for the Euclidean norm and $^\top$ for the transpose of vectors, we have
   \begin{equation}\label{eq:multK}
   K(s,t)=\exp\left(-\frac{\|s-t\|^2}{2}\right)-\left\{1+s^\top t+\frac{(s^\top t)^2}{2}\right\}\exp\left(-\frac{\|s\|^2+\|t\|^2}{2}\right),\quad s,t\in\R^d.
   \end{equation}
   The first step is to derive explicit expressions for eigenvalues w.r.t. the kernel $K_0(s,t)=\exp\left(-\|s-t\|^2/2\right)$, $s,t\in\R^d$. The second step is to find the corresponding multivariate representation of \eqref{eq:kernelrep}, which seems to be non-standard, since the quadratic summand $(s^\top t)^2$ in \eqref{eq:multK} does not factorize easily. Both problems have to be solved in order to successfully apply the method presented in Section \ref{sec:fredholm}.

\bibliography{lit_BHEP_EV}   
\bibliographystyle{abbrv}      

\vspace{5mm}
\noindent
B. Ebner and N. Henze, \\
Institute of Stochastics, \\
Karlsruhe Institute of Technology (KIT), \\
Englerstr. 2, D-76133 Karlsruhe. \\
E-mail: {\texttt Bruno.Ebner@kit.edu}\\
E-mail: {\texttt Norbert.Henze@kit.edu}

\end{document}